\documentclass[10pt,a4paper,oneside]{article}
\usepackage[english]{babel}
\usepackage{a4wide,amsmath,amsthm,amssymb,url,graphicx,xspace,pgf,enumerate}

\usepackage{tikz}
\usetikzlibrary{calc,through,backgrounds,arrows}

\DeclareMathOperator{\wt}{wt}

\newcommand{\vek}[1]{\boldsymbol{#1}}

\theoremstyle{definition}
\newtheorem{theorem}{Theorem}
\newtheorem{lemma}[theorem]{Lemma}

\date{}

\author{
Ivan Landjev
Konstantin Vorobev
\thanks{This research was supported 
	by the NSP P. Beron project CP-MACT.}
}

\title{An upper bound on the size of a code with $s$ distances}
\begin{document}
\maketitle

\begin{abstract}
Let $C$ be a binary code of length $n$ with distances
$0<d_1<\cdots<d_s\le n$.  
In this note we prove 
a general upper bound on the size of $C$
without any restriction on the distances $d_i$.
The bound is asymptotically optimal. 
\end{abstract}



\section{Introduction}
\label{sec:intro}

A binary (nonlinear) code
$C\subset\mathbb{F}_2^n$ is called an $s$-distance code if the possible distances
between two different words of the code take on $s$ different values, i.e.
\[\{d(u,v)\mid u,v\in C, u\ne v\}=\{d_1,d_2,\ldots,d_s\},\]
where $0<d_1<d_2<\ldots,d_s\le n$. 
Here $d(u,v)$ denotes the Hamming distance between $u$ and $v$.
A binary code of length $n$, cardinality $M$, and with distances from
the set $\mathcal{D}$
is called an $(n,M,\mathcal{D})$-code. If the cardinality is not specified we speak of an $(n,\mathcal{D})$-code. A natural problem is to 
determine the maximal cardinality, denoted by $A_2(n,\mathcal{D})$, of a binary code of fixed length $n$ and distances in the set $\mathcal{D}$.
 
In a recent paper by Barg et al. \cite{Barg-et-al} it was proved that 
\begin{equation}
\label{eq:barg}
A_2(n,\{d_1,d_2\})\le {n\choose2}+1,
\end{equation}
for $n\ge6$. There is no restriction on the distances $d_1$ and $d_2$.
This bound is achieved, e.g. we have $A_2(n,\{2,4\})={n\choose2}+1$ for all $n\ge6$
\cite{BDZZ21,LRV23}. The proof of \cite{Barg-et-al} is based on the idea to embed a binary code into a sphere $S^{n-1}$ and relate its size to spherical two distance sets
and spherical equiangular sets, using some recently proved bounds.

For codes with more than two weights, we have the following general results.

\begin{theorem}(A. Barg, O. Musin, 2011)\cite{BM11}
Let $\mathcal{D}=\{d_1,\ldots,d_s\}$, where 
$\sum_i d_i\le \frac{1}{2}sn$. Then
\[A_2(n,\mathcal{D})\le \sum_{i=0}^{s-2} {n\choose i}+ {n\choose s}.\]
\end{theorem}

\begin{theorem}(H. Nozaki, M. Shinohara, 2010)\cite{NS10}
Let $\mathcal{D}=\{d_1,\ldots,d_s\}$ and let
\[f(t)=\prod_i\frac{d_i-t}{d_i}=\sum_{k=0}^s f_k\phi_k(t)\]
where
\[\phi_k(x)=\sum_{j=0}^k (-1)^j{x\choose j}{n-x\choose k-j}, k=0,\ldots,n\]
are the Krawtchouk polynomials. Then
\[A_2(n,\mathcal{D})\le \sum_{k: f_k>0} {n\choose k}.\]
\end{theorem}

Both results require some restrictions on the distances in the 
$\mathcal{D}$. In this note, we prove a bound on $A_2(n,\mathcal{D})$ without
restriction on the distances in $\mathcal{D}$.

\section{The Upper Bound}
\label{sec:2}

Let $C$ be a binary $(n,M,\mathcal{D})$-code with distances
$0<d_1<d_2<\ldots<d_s\le n$.
Consider the ring of polynomials 
$\mathbb{R}[x_1, x_2, \dots, x_n]$ over the field of the real numbers. 
Let us denote $\vek{x}=(x_1, x_2, \dots, x_n)$.
For a vector $\vek{v}=(v_1, v_2, \dots, v_n)$ over 
$\mathbb{R}$, define the polynomial 
\[P_{\vek{v}}(\vek{x})=\prod_{i=1}^s
	\big(d_i-(\vek{x},\vek{1})-(\vek{v},\vek{1})+2(\vek{v},\vek{x}) \big),\]
Here $(\vek{x},\vek{y})=\sum_i x_iy_i$ is the usual dot product 
(over $\mathbb{R}$) of $\vek{x}$ and $\vek{y}$.

Let us note that
\[2(\vek{v},\vek{x})-(\vek{x},\vek{1})=
\pm x_1\pm x_2\pm\cdots\pm x_n,\]
where the sign of $x_i$ is "$+$" if $v_i=1$, and 
"$-$" if $v_i=0$. If we set
\[\sigma_{\vek{v}}(j)=\left\{
\begin{array}{ll}
0 & \text{ if } v_j=1, \\
1 &	\text{ if } v_j=0,
\end{array}\right.\]
we get
\[2(\vek{v},\vek{x})-(\vek{x},\vek{1})=
\sum_{j=1}^n (-1)^{\sigma_{\vek{v}}(j)}x_j.\]
For a vector $\vek{v}$ of Hamming weight $d_k$, i.e. $\wt(\vek{v})=d_k$: 
\begin{equation}\label{eq:pold1}
P_{\vek{v}}(\vek{x})=
\prod_{i=1}^s
\left(d_i-d_k+\sum_j (-1)^{\sigma_{\vek{v}}(j)}x_j \right),
\end{equation}

\begin{lemma}
\label{lma:1}	
Let $\vek{u},\vek{v}\in C$. Then
\[P_{\vek{u}}(\vek{v})=\left\{
\begin{array}{cl}
\prod_{i=1}^s d_i  & \text{ if } \vek{u}=\vek{v}, \\
0                  & \text{ if } \vek{u}\ne\vek{v}.
\end{array}
\right.\]
\end{lemma}

\begin{proof}
As usual, for the vectors $\vek{u}=(u_1,\cdots,u_n)$ and $\vek{v}=(v_1,\cdots,v_n)$,
we define $\vek{u}*\vek{v}=(u_1v_1,\cdots,u_nv_n)$.
Then we have
\begin{eqnarray*}
P_{\vek{u}}(\vek{u}) &=& \prod_{i=1}^s (d_i-(\vek{u},\vek{1})-(\vek{u},\vek{1})+
2(\vek{u},\vek{u})) \\	
 &=& \prod_{i=1}^s (d_i-\wt(\vek{u})-\wt(\vek{u})+2\wt(\vek{u})) \\
 &=& \prod_{i=1}^s d_i,
\end{eqnarray*}
and
\begin{eqnarray*}
	P_{\vek{u}}(\vek{v}) &=& \prod_{i=1}^s (d_i-(\vek{u},\vek{1})-(\vek{v},\vek{1})+
	2(\vek{u},\vek{v}) \\	
	&=& \prod_{i=1}^s (d_i-\wt(\vek{u})-\wt(\vek{v})+2\wt(\vek{u}*\vek{v})) \\
	&=& \prod_{i=1}^s (d_i-d(\vek{u},\vek{v})).
\end{eqnarray*}
Since $d(\vek{u},\vek{v})=d_i$ for some $i\in\{1,\cdots,s\}$ we have
$P_{\vek{u}}(\vek{v})=0$.
\end{proof}

\begin{lemma}
\label{lma:2}
Let $C$ be a binary code with distances $0<d_1,\dots<d_s\le n$.
The polynomials in the set
\[\{P_{\vek{u}}(\vek{x})\mid \vek{u}\in C\}\]
are linearly independent.
\end{lemma}

\begin{proof}
Assume that 
$\displaystyle \sum_{\vek{u}\in C} c_{\vek{u}}P_{\vek{u}}(\vek{x})=0$ for some  real numbers $c_{\vek{u}}$,
$\vek{u}\in C$. Plug in $\vek{x}=\vek{v}$ for an arbitrary codeword $\vek{v}\in C$.
By Lemma~\ref{lma:2}, we get
\[0=\sum_{\vek{u}\in C} c_{\vek{u}}P_{\vek{u}}(\vek{v})=c_{\vek{v}}P_{\vek{v}}(\vek{v})=
c_{\vek{v}}\prod_{i=1}^s d_i.\]
This implies $c_{\vek{v}}=0$ for all $\vek{v}\in C$, which proves the Lemma.
\end{proof}

\begin{theorem}
\label{thm:1}
Let $C\subset\{0,1\}^n$ be a binary code with distances $0<d_1<\cdots<d_s\le n$. Then
\[|C|\le {n+s\choose s}.\] 
\end{theorem}

\begin{proof}
Let $V$ be the subspace of $\mathbb{R}[x_1,\cdots,x_n]$ of all polynomials of degree at most $s$, and $U$ -- the subspace spanned by 
$\{P_{\vek{u}}(\vek{x})\mid \vek{v}\in C\}$. Clearly $U<V$, whence 
$\dim U\le\dim V$.

By Lemma~\ref{lma:2}, $\dim U=|C|$. On the other hand, $V$ is spanned
by the monomials $x_1^{\alpha_1}x_2^{\alpha_2}\ldots x_n^{\alpha_n}$, where
	$\alpha_1+\alpha_2+\cdots+\alpha_n\le s$
The number of the monomials of degree $j$ is equal to the number of the solutions of
$y_1+y_2+\cdots+y_n=j$ in non-negative integers which is 
${n+j-1\choose j}$. Thus
\[\dim V=\sum_{j=0}^s {n+j-1\choose j}={n+s\choose s}.\]
Now 
\[|C|=\dim U\le\dim V={n+s\choose s}.\]
\end{proof}

I should be noted that this bound has been proved in \cite{Bannai-et-al,Blok82,Blok83} for $s$-distance sets on the Euclidean sphere. This result then easily implies Theorem~\ref{thm:1}. In the next section, we give an improvement on the bound from 
Theorem~\ref{thm:1}. 

Theorem~\ref{thm:1} says that if $C$ is a code with $s$ distances 
then $|C|=O(n^s)$.
It has to be noted that this bound is asymptotically tight. Consider the code $C$ consisting of all words of Hamming weight $s$. If $\vek{u}$ and $\vek{v}$ 
are two words from $C$ with $wt(\vek{u}*\vek{v})=i$, 
then $d(\vek{u},\vek{v})=2(s-i)$, where $i\in\{0,\cdots,n-1\}$.
Hence $C$ is a code with $s$ distances: $2,4,\cdots,2s$. 
If $s$ is even the zero word can be adjoined to $C$.
In many cases this construction is optimal, as in the case of two weights
(cf. \cite{Barg-et-al,BDZZ21,BDZZ22,LRV23}).

Note that the blocks of every design with $s$ intersection numbers 
form a code with $s$ intersection numbers. In particular, quasi-symmetric designs give 
rise to codes with two distances. If a 2-distance code has the additional property that there is a word at equal distance from all codewords then the construction described above is the best possible. The following theorem is well-known \cite{CvL80} (Proposition 3.5).

\begin{theorem}
If a quasi-symmetric design does not contain repeated blocks then
\[b\le\frac{v(v-1)}{2}.\]
\end{theorem}

If we exclude the trivial examples, we have the following strengthening \cite{CvL80}
(Proposition 3.6).

\begin{theorem}
Let $\mathcal{D}$ be a quasi-symmetric design without repeated blocks and with $4\le k\le v-4$. Then $\displaystyle b\le {v\choose 2}$ with equality 
if and only if $\mathcal{D}$ is a 4-design.
\end{theorem}

In fact, up to taking the complement, there exists exactly one design for which we have equality in the above theorem. This is the classical 4-$(23,7,1)$-design with intersection numbers 1  and 3. Its blocks form a $(23,253,\{8,12\})$-code.

For codes $C$ of length $n$ with $s$ distances, $s>2$, it is expected that
$|C|\le{n\choose s}$ or ${n\choose s}+1$ (depending on whether $s$ is odd or even)
with equality achieved only for the trivial construction.

\section{An Improvement on the Upper Bound}

The bound from Theorem~\ref{thm:2} can be improved by inspecting more closely the dimension of the vector space $U$. In what follows, let $C$ be a binary code of length $n$ with $s$ distances: $0<d_1<\cdots<d_s\le n$. 

Consider the coefficient of the monomial 
$x_1^{2\alpha_1}x_2^{2\alpha_2}\cdots x_n^{2\alpha_n}$
in the polynomial $P_{\vek{v}}(\vek{x})$ where $\wt(\vek{v})=d_k$ for some $k$. 
If $t=\sum_i\alpha_i$, then $2t\le s$. 
From (\ref{eq:pold1}) we can deduce that this coefficient is equal to
\[{2t\choose 2\alpha_1}{2t-2\alpha_1\choose 2\alpha_2}\cdots{2t-2\alpha_1-\cdots-2\alpha_{n-1}\choose 2\alpha_n}
\sum_{\{i_1,\cdots,i_{s-2t}\}}(d_{i_1}-d_k)(d_{i_2}-d_k)\ldots(d_{i_{s-2t}-d_k})\]
where the sum is taken over all $(s-2t)$-tuples of indices $i_1,\cdots,i_{s-2t}$
contained in $\{1,\cdots,s\}$. Thus the coefficients of 
$x_{\sigma(1)}^{2\alpha_1}x_{\sigma(2)}^{2\alpha_2}\cdots x_{\sigma(n)}^{2\alpha_n}$ are the same for all permutations $\sigma$ of $\{1,\cdots,n\}$.

This implies the following improvement for the bound of Theorem~\ref{thm:1}.

\begin{theorem}
	\label{thm:2}
	For every binary code $C$ of length $n$ with $s$ distances, it holds
	\[|C|\le {n+s\choose s}-\sum_{t:2t\le s} \left({t+n-1\choose t}-p(t)\right).\]
	where $p(t)$ is the number of the partitions of $t$.	
\end{theorem}

\begin{proof}
Two solutions $(\alpha_i)$ and $(\beta_i)$
of the equation 
\begin{equation}
\label{eq:eq1}
y_1+\cdots+y_n=t, \;\;y_i\ge0
\end{equation}
are said to be equivalent if one of them is a permutation of the other,
i.e. $\beta_=\alpha_{\sigma(i)}$ for some permutation $\sigma$.
The number of the equivalent classes is equal to $p(t)$, the number of partitions of $t$.

From our preliminary observation before the theorem
we have that for the equivalent solutions $(\alpha_i)$
and $(\beta_i)$ the coefficients of
$x_1^{2\alpha_1}x_2^{2\alpha_2}\cdots x_n^{2\alpha_n}$
and $x_1^{2\beta_1}x_2^{2\beta_2}\cdots x_n^{2\beta_n}$
are equal in every polynomial $P_{\vek{v}}(\vek{x})$.	

Let $U$ and $V$ be the same as in Theorem~\ref{thm:2}. 
We have $V=V_1\oplus V_2$, where
\begin{eqnarray*}
V_1 &=& \langle x_1^{\alpha_1}\cdots x_n^{\alpha_n} \mid \alpha_i \text{ even for all } i\rangle, \\
V_2 &=&	\langle x_1^{\beta_1}\cdots x_n^{\beta_n} \mid \text{ at least one } \beta_i \text{ odd } \rangle
\end{eqnarray*}
Let $U_i=V_i\cap U$. We have proved that
we have that
\[U_1\subseteq\langle \sum_{\alpha} \vek{x}^{\alpha} \ \alpha \text{ runs over all classes
	of equivalent solutions of } (\ref{eq:eq1}) \rangle.\]
Now we have
\begin{eqnarray*}
\dim V_2 &=& \dim V - \dim V_1 \\
         &=& {n+s\choose s} - \sum_{t:t\le s/2} \#(\text{solutions of } (\ref{eq:eq1}))\\
         &=& {n+s\choose s} - \sum_{t:t\le s/2} {t+n-1\choose t}.
\end{eqnarray*}
On the other hand, 
\[\dim V_1\le\sum_{t:t\le s/2} \#(\text{equivalence classes of solutions of }(\ref{eq:eq1}))
=\sum_{t:t\le s/2} p(t).\]
Now
\begin{eqnarray*}
\dim U &=& \dim U_1+\dim U_2 \\
       &\le& \dim V_1+\dim V_2 \\
       &=&  \sum_{t:t\le s/2} p(t) + {n+s\choose s} - \sum_{t:t\le s/2} {t+n-1\choose t}
\end{eqnarray*}
which was to be proved.
\end{proof}

This bound is also not sharp. It is expected that the largest codes with $s$ distances are isomorphic to those  of size $\displaystyle {n\choose s}$ or 
$\displaystyle {n\choose s}+1$ described in section~\ref{sec:2}.

\noindent
{\sl Institute of Mathematics and Informatics}\\
{\sl Bulgarian Academy of Sciences}\\
{\sl 8 Acad. G. Bonchev str.} \\
{\sl 1113 Sofia, Bulgaria}\\
{\sl e-mail:} \texttt{ivan@math.bas.bg}


\noindent
{\sl Institute of Mathematics and Informatics}\\
{\sl Bulgarian Academy of Sciences}\\
{\sl 8 Acad. G. Bonchev str.}\\
{\sl 1113 Sofia, Bulgaria}\\
{\sl e-mail:} \texttt{konstantin.vorobev@gmail.com}
\end{document}